\documentclass[11pt]{amsart}
\usepackage{amsthm}
\usepackage{amssymb,amsmath,amsfonts,mathrsfs}
\usepackage{epsfig,subfigure,graphicx}
\usepackage{ifthen}
\usepackage{colortbl}

\title[Second-order Kuramoto oscillators]
{On the region of attraction of phase-locked states for swing equations  on connected graphs with  inhomogeneous dampings}

\author[Choi]{Young-Pil Choi}
\address[Young-Pil Choi]{\newline Fakult\"at f\"ur Mathematik
    \newline  Technische Universit\"at M\"unchen, Boltzmannstra{\ss}e 3, 85748, Garching bei M\"unchen, Germany}
\email{ychoi@ma.tum.de}

\author[Li]{Zhuchun Li}
\address[Zhuchun Li]{\newline Department of Mathematics \newline  Harbin Institute
of Technology, Harbin 150001, China
}
\email{lizhuchun@hit.edu.cn}
\newtheorem{theorem}{Theorem}[section]
\newtheorem{lemma}{Lemma}[section]

\newtheorem{proposition}{Proposition}[section]
\newtheorem{remark}{Remark}[section]

\newtheorem{definition}{Definition}[section]

\newcommand{\bbr}{\mathbb R}

\newcommand{\lal}{\langle}
\newcommand{\ral}{\rangle}
\newcommand{\e}{\varepsilon}
\newcommand{\lt}{\left}
\newcommand{\rt}{\right}
\newcommand{\me}{\mathcal{E}}
\newcommand{\md}{\mathcal{D}}
\newcommand{\bq}{\begin{equation}}
\newcommand{\eq}{\end{equation}}
\newcommand{\wt}{\widetilde}

\def\charf {\mbox{{\text 1}\kern-.30em {\text l}}}

\begin{document}

\date{\today}

\subjclass{34C15, 34D06, 92D25} \keywords{Synchronization, region of attraction,  transient stability, second-order Kuramoto oscillators,  lossless power grids,   connected network, inhomogeneous dampings}


\begin{abstract}
We consider the synchronization problem of swing equations, a second-order Kuramoto-type model, on   {\em connected networks} with {\em inhomogeneous dampings}. This was largely motivated by its relevance to the dynamics of  power grids.    We   focus on the   estimate of  the region of attraction of  synchronous states which is  a central problem in the transient stability of power grids.  In the  recent
literature, D\"{o}rfler, Chertkov, and Bullo [{\em Proc. Natl. Acad. Sci. USA}, 110 (2013), pp. 2005-2010]
found a condition for the synchronization in smart grids. They pointed out that the region of attraction  is an important unsolved problem. In  [{\em SIAM J. Control Optim.}, 52 (2014), pp. 2482-2511],  only a special case was considered where the oscillators have homogeneous dampings and the underlying graph has a diameter less than or equal to 2. There the  analysis heavily relies on these   assumptions; however, they are too strict compared to the real power networks. In this paper, we continue the study and  derive an estimate on the region of attraction of phase-locked states for  lossless  power grids on connected graphs with inhomogeneous  dampings. Our main strategy is based on the gradient-like formulation and energy estimate. We refine the assumptions by constructing a new energy functional which enables us to consider such   general settings.
\end{abstract}
\maketitle \centerline{\date}

\section{Introduction}
\textbf{General background.-}
The synchronization  of large populations of weakly coupled oscillators is very common in nature, and it has been extensively studied in various scientific communities such as physics, biology,  sociology, etc. The scientific interest in the synchronization of coupled oscillators can be traced back to Christiaan Huygens' report on coupled pendulum clocks \cite{H}.  However, its rigorous mathematical treatment was done by Winfree \cite{W} and Kuramoto \cite{K} only several decades ago. Since then, the Kuramoto model became a   paradigm for synchronization  and various extensions have been extensively explored in scientific communities such as applied mathematics \cite{C-H-J-K,C-H-Y,C-L-H-X-Y}, engineering and control theory \cite{C-S, D-B-1, D-B-0, D-B-2}, physics \cite{A-B,  P-R-K, S}, neuroscience and biology \cite{E, K}.

In the present work, we consider the synchronization of a variant of Kuramoto model which has relevant significance in engineering, in particular, the  power grids with general network topology and inhomogeneous dampings.
  The power grid,  as a complex large-scale system, has  rich nonlinear dynamics, and its synchronization and transient stability are very important   in real applications.
The    transient stability, roughly speaking, is
concerned with the ability of a power network  to settle into an acceptable
steady-state operating condition following a large
disturbance.  In recent years, renewable energy has fascinated  not only the scientific community but also the industry. It is believed that  the future power generations  will
rely increasingly on renewables such as wind and solar power, and the industry of renewable power has been in growth. These renewable
power sources are highly stochastic; thus,   an increasing number of transient
disturbances  will act on   increasingly complex power grids. As a consequence, it becomes  significantly important    to study  complex power grids and their transient stability.


\textbf{Literature review.-} The similarity between the  power grids and  nonuniform  second-order (inertial)  Kuramoto oscillators  \begin{equation*}\label{eq1}m_i\ddot{\theta}_i+d_i\dot{\theta}_{i} = \Omega_{i} + \sum_{j=1}^{N} a_{ij} \sin(\theta_{j} - \theta_{i})\end{equation*}
has been reported and explored  in  many literature such as   \cite{D-B-2, F-N-P,F-R-C-M-R,S-U-S-P}. If we  take $m_i=0$, $d_i=1$ and $a_{ij}=K/N$, then it reduces to the classic Kuramoto model   with mean-field coupling strength $K$.  The synchronization of the classic model has been studied in many literature, such as \cite{C-H-J-K, C-S, J-M-B, L-X-Y,   V-M,V-M2}.
   This problem is to look
for  conditions on the parameters and/or  initial phase configurations   leading to the existence  or emergence of  phase-locked states.
 The inertial effect   was first conceived
by Ermentrout \cite{E} to explain the slow synchronization of certain biological systems.  Mathematically, incorporating
the inertial effect into Kuramoto  model  is simply adding the second-order term,  resulting in a model with $m_i=m, d_i=1, a_{ij}=K/N$,
which  causes richer phenomena from the dynamical viewpoint.   For   mathematical results on the  inertial model we refer to \cite{C-H-Y,C-L-H-X-Y,D-B-1,L-X-Y1}.  A connection between first and second-order models is the topological  conjugacy   argument  in \cite{D-B}.


The power networks with synchronous motors can be described by swing equations, a system of nonuniform second-order Kuramoto oscillators, see Subsection \ref{subsecmodels}.   The transient stability, in terms of power grids,   is concerned
with the system's ability to reach an acceptable  synchronism after
a major disturbance such as short circuit caused by
lightning. 
 Then the fundamental problem, as pointed in the survey \cite{V-W-C},  is: {\em whether the post-fault  state (when the disturbance is cleared) is located in  the region of attraction  of   synchronous states}. Thus, a closely related issue is to estimate the region of attraction  of   synchronous states.    In the recent paper \cite{D-C-B},     the authors focused on the network topology, but as the authors mentioned, {\it ``another important question
not addressed in the present article concerns the region
of attraction of a synchronized solution''}.  Therefore, the region of attraction  of synchronized states is indeed a central problem for the transient stability.

 For the power grid,
some analysis on transient stability  can be found in \cite{C, C-C-C, V-W-C}, where the approach is the so-called direct method based on the   energy function.   However, this method did not provide explicit formulas to check if the power
system synchronizes for   given initial data and parameters.  Actually, the    energy function, containing a  pair-wise nonlinear attraction with terms $\cos(\theta_i-\theta_j)$,  is difficult to study.  Another tool is based on the  singular perturbation theory \cite{C-W-V, D-B-1} by which the second-order dynamics can be approximated by the first-order
dynamics when the system is sufficiently strongly over-damped, i.e.,  the ratio of inertia over damping  is sufficiently
small.   
For example, in \cite{D-B-1}  the authors studied the more sophisticated power networks  with energy losses (phase shifts)  and  derived algebraic conditions
that relate the synchronization   to the underlying network structure.
  Unfortunately there is no formula  in \cite{C-W-V,D-B-1}    to check whether a given system is so strongly over-damped that the result can be applied.   In the  survey paper \cite{D-B-0}, D\"{o}rfler and Bullo pointed out that the transient
dynamics of second-order oscillator networks is  a  challenging open problem.

As far as the authors know, the direct analysis on the region of attraction for second-order Kuramoto oscillators could be found only in \cite{C-H-Y, C-L-H-X-Y, L-X-Y1}. However, in terms of power grids,  there are drawbacks in at least two aspects. First, in these studies the inertia and damping are assumed to be either uniform \cite{C-H-Y, C-L-H-X-Y}  or homogeneous \cite{L-X-Y1}, which is not realistic in   power generators.  
   The second one lies in the network topology. For example,  in \cite{L-X-Y1}  the transient stability was considered  when the underlying graph  have a diameter less than or equal to 2.  In  \cite{D-B-1},  the  underlying network has to be even all-to-all (see \cite[Theorem 2.1]{D-B-1}).   In practice, it is not realistic to assume  that a power network should have such a nice connectivity,    
   for example,   the Northern European power grid   \cite{M-H-K-S}. 
   Thus, the real situation challenges us to consider   the general systems with inhomogeneous  dampings and general networks.  

\textbf{Contributions.-}
The main contribution of this paper is to  {\em estimate the region of attraction of synchronous states for lossless  power grids on general  networks with inhomogeneous dampings}.  To the best of the authors' knowledge, this is the first rigorous study on this challenging problem for such a general model of lossless power grids with oscillators.  We use a direct analysis on
the dynamics of second-order Kuramoto-type model and derive an {\em explicit} formula to
estimate the   region of attraction.

{Among the rigorous  analysis of Kuramoto oscillators, a typical  method is to study the dynamics of phase difference, for example, \cite{C-H-J-K, C-H-Y, C-L-H-X-Y, C-S, D-B, L-X-Y1}. However, such an analysis crucially  relies on the homogeneousness of parameters and the nice connectivity that the diameter of the  graph should be less than or equal to 2. Thus, this method fails for the current case. Our strategy is to use the gradient-like formulation  and  energy method.   Departing from the (physical) energy   for the so-called direct method in \cite{C, C-C-C, V-W-C}, we will construct a virtual energy function which enables us to derive the boundedness of the trajectory.  Then we can use the {\L}ojasiewicz's theory   to derive the convergence immediately. We also remark that our virtual  energy    is  different  with that in \cite{C-L-H-X-Y} where the uniform inertia and damping were considered.}




\textbf{Organization of paper.-}  In Section 2, we present the models, main result and some discussions.     In Section 3, we give a proof to the main result.  In Section 4, we present some numeric illustrations. Finally, Section 5 is devoted to a conclusion. 


{\textbf{Notation:}}

 \noindent $\|\cdot\|$---Euclidean norm in $\mathbb R^N$,



\noindent $L^\infty(\mathbb R^+,\mathbb R^N)=\left\{f:\mathbb R^+\rightarrow \mathbb R^N\mid f \,\,\mbox{is bounded}\right\},$   

\noindent $W^{1,\infty}(\mathbb R^+,\mathbb R^N)=\left\{f:\mathbb R^+\rightarrow \mathbb R^N\mid f \,\,\mbox{is differentiable}, f, f'\in  L^\infty(\mathbb R^+,\mathbb R^N)\right\}.$


\section{Models, main result and discussions}\label{preliminaries}
\setcounter{equation}{0}
In this section, we  present the model of power grids as a second-order Kuramoto-type model, and  its gradient-like flow formulation together with   a key convergence result for the  general gradient-like
system with analytic potentials. Some preliminary   inequalities  are also provided.

\subsection{Models}\label{subsecmodels}
A mathematical model for a {\em lossless}   network-reduced power system \cite{C-C-C,D-C-B}  can be defined by the following swing equations:
 \begin{equation} \label{grid}
m_i\ddot{\theta}_i+d_i\dot{\theta}_{i} = P_{m,i} + \sum_{j=1}^{N} |V_i|\cdot|V_j|\cdot\Im(Y_{ij}) \sin(\theta_{j} - \theta_{i}), \quad i=1,2,\cdots,N, \quad t > 0.
\end{equation}
 Here $\theta_i $ and $\dot{\theta}_{i}$ are the rotor
 angle and frequency of the $i$-th generator, respectively. The parameters $P_{m,i}>0$, $|V_i|>0$, $m_i>0$, and $d_i>0$ are the effective power input, voltage level, generator
inertia constant, and damping coefficient of the $i$-th generator, respectively. For $Y=(Y_{ij})$ we denote the  symmetric
nodal admittance matrix, and $\Im(Y_{ij})$ represents
the susceptance of the transmission line between $i$ and $j$. If the power network is subject to energy loss   due to the transfer conductance, then it should be depicted by a phase shift in each coupling term \cite{D-B-1}.   We   refer to \cite{D-B-1,D-C-B,S-P}  for more details or the
 derivation of \eqref{grid} from physical principles.
For simplicity in mathematical sense,  let us take $\Omega_{i}=P_{m,i}$ and $a_{ij}=|V_i|\cdot|V_j|\cdot\Im(Y_{ij})$, and drop the hats in \eqref{grid}. Then the system \eqref{grid} becomes   a second-order   model of coupled oscillators 
\begin{align}\begin{aligned}\label{Ku-iner-net}
m_i\ddot{\theta}_i+d_i\dot{\theta}_{i}& = \Omega_{i} + \sum_{j=1}^{N} a_{ij} \sin(\theta_{j} - \theta_{i}),\quad i=1,2,\dots,N. 
\end{aligned}\end{align}
Here, the coupling between oscillators is symmetric since $Y$ is a symmetric matrix.
If $m_i/d_i=m_j/d_j$ for all $i\neq j$, it is said to be a model with {\em homogeneous} dampings.
 We can define a graph $\mathcal G=(\mathcal V, \mathcal W)$ associated to the system \eqref{Ku-iner-net}  such that   $\mathcal V=\{1,2,\dots, N\},$ and  $\mathcal W=\left\{(i,j): a_{ij}>0\right\}.$  In this setting,  we call $\mathcal G$  the  {undirected}  graph  induced by the  matrix $A=(a_{ij})$.

  {We acknowledge that a real power network should contain both generators and loads, while  the above  model includes only generators.  In power flow, loads can be modeled by different ways, for example, a system of first-order Kuramoto oscillators \cite{D-C-B} or algebraic
equations.  Another typical way is to use the Kron reduction to    obtain so-called ``network-reduced''  model  so that the loads  are involved in the  transfer admittance \cite{D-B-2, Ward}, and the resulted system consists of only generators.  In such a sense, the network-reduced model \eqref{grid} becomes an often studied mathematical model for power grids.  It is worthwhile to mention that the Northern European power grid  in \cite{M-H-K-S} does not have the nice connectivity in literature \cite{D-B-1,L-X-Y1}  after the so-called Kron reduction \cite{D-B-2}  (this can be seen by looking into the  power flow chart in \cite[Fig.4]{M-H-K-S} together with the topological properties of Kron reduction in \cite[Theorem III.4]{D-B-2}). }  

Next, we recall some definitions for complete synchronization of coupled oscillators.

\begin{definition}
Let $\theta(t)=(\theta_1(t), \dots, \theta_N(t))$ be an ensemble of phases of Kuramoto oscillators.
\begin{enumerate}
\item
The Kuramoto ensemble asymptotically exhibits complete   frequency  synchronization if and only if
\[ \displaystyle \lim_{t \to \infty} |\omega_i(t) - \omega_j(t)| = 0,  
\quad \forall\; i \not = j. \]
Here, $\omega_i(t):=\dot\theta_i(t)$ is the frequency of $i$th oscillator at time $t$.

\item
The Kuramoto ensemble asymptotically exhibits phase-locked state if and only if the relative phase differences converge to some constant asymptotically:
\[ \displaystyle \lim_{t \to \infty} ({\theta}_i(t) - {\theta}_j(t) )= \theta_{ij}, \qquad \forall\; i \not = j. \]
\end{enumerate}
\end{definition}
\subsection{A macro-micro decomposition}\label{s3.1}
We notice that the system \eqref{Ku-iner-net} can be rewritten as a system of first-order ODEs:
\begin{align*}
\begin{aligned}
{\dot \theta}_i &= \omega_i, \quad i=1, 2, \dots, N, \quad t > 0, \\
{\dot \omega}_i &= \frac{1}{m_i} \left[ -d_i\omega_i +   \Omega_i +  \sum_{j=1}^{N}
a_{ij}\sin(\theta_j - \theta_i) \right].
\end{aligned}
\end{align*}
Let $\theta:=(\theta_1, \theta_2, \dots, \theta_N)$, \,$\omega:=(\omega_1,\omega_2, \dots,\omega_N)$, \,$M:=diag\{m_1, m_2, \dots, m_N\}$, $D:=diag\{d_1, d_2, \dots, d_N\}$,  and  $ \Omega:=( \Omega_1,  \Omega_2, \dots,  \Omega_N)$. Using these newly defined notations, we  introduce macro variables as follows:
\begin{equation}\label{vs}
\Omega_c:= \frac{\sum_{i=1}^N\Omega_i}{tr(D)} = \frac{\sum_{i=1}^N\Omega_i}{\sum_{i=1}^N d_i}, \quad \theta_s:=   \sum_{i=1}^Nd_i\theta_i, \quad \omega_s:=  \sum_{i=1}^Nm_i\omega_i,
\end{equation}
where $tr(\cdot)$ denotes the trace of a matrix. We also set the phase fluctuations   (micro-variables) as $$\hat\theta_i:=\theta_i- \Omega_c \,t, \quad\,i=1,2,\dots,N,$$  then we get $\ddot{\hat\theta}_i=\ddot{\theta}_i$, $\dot{\hat\theta}_i=\dot{\theta}_i-\Omega_c\,$, and the system \eqref{Ku-iner-net} can be rewritten as
 \begin{equation} \label{Ku-iner-net1}
m_i\ddot{\hat\theta}_i+d_i\dot{\hat\theta}_{i} = \hat\Omega_i + \sum_{j=1}^{N} a_{ij} \sin(\hat\theta_{j} - \hat\theta_{i}) \quad \mbox{with} \quad \hat\Omega_i := \Omega_{i}-d_i \Omega_c,
\end{equation}
where   the ``micro'' natural frequencies $\hat\Omega_i$ sum to zero:
\begin{equation*}
\sum_{i=1}^N\hat \Omega_i=0.
\end{equation*}
In particular, if $\Omega_i/d_i= \Omega_j/d_j$ for all $i,j=1,2,\dots,N$, then we have $\hat\Omega_i=0$ for each $i$ and the equation \eqref{Ku-iner-net1} reduces to a system of coupled  oscillators with identical natural frequencies:
\begin{equation*}\label{Ku-iner-homo} m_i\ddot{\hat\theta}_i+d_i\dot{\hat\theta}_{i} = \sum_{j=1}^{N} a_{ij} \sin(\hat\theta_{j} - \hat\theta_{i}).
\end{equation*}
 {Note that the   ensemble of micro-variables $(\hat\theta_1, \dots, \hat\theta_N)$ is a phase shift of the original ensemble $(\theta_1, \dots, \theta_N)$, thus, they share the same asymptotic property  as long as we concern only the  synchronization or phase-locking  behavior.  Moreover, the equations for the variable $\theta_i$ and $\hat\theta_i$, i.e., \eqref{Ku-iner-net} and \eqref{Ku-iner-net1}, have the same form. So, we may consider \eqref{Ku-iner-net1} instead of \eqref{Ku-iner-net} when we concern the synchronization problem.
 These observations enable us to assume,  without loss of generality,   the natural frequencies in \eqref{Ku-iner-net} satisfy
\begin{equation}\label{eqsum0}
\sum_{i=1}^N\Omega_i=0.
\end{equation}
In the rest of this paper, we consider the system \eqref{Ku-iner-net} with \eqref{eqsum0}.}



\subsection{An   inequality on connected graphs}
Consider a symmetric and connected network, which can be realized with a weighted graph $\mathcal G = (\mathcal V, \mathcal W, A)$. Here, $\mathcal V = \{1, 2, \dots, N\}$ and $\mathcal W \subseteq \mathcal V\times \mathcal V$ are vertex and edge sets, respectively, and $A = (a_{ij})$ is an $N \times N$ matrix whose element $a_{ij}$ denotes the capacity of the edge (communication weight) flowing from $j$ to $i$. We note that the underlying network  of  power grids \eqref{Ku-iner-net} is  {\em undirected},  i.e., the adjacency matrix $A=\{a_{ij}\}$ is  symmetric.  We say the graph $\mathcal G$ is connected  if for any pair of nodes  $i,j\in \mathcal V$, there exists a shortest  path from $i$ to $j$, say
\[i=p_1\to p_2\to p_3\to \cdots \to p_{ _{d_{ij}}}=j, \quad (p_k, p_{k+1})\in \mathcal W, \quad k=1,2,\dots,d_{ij}-1.\]  In order  for  the complete synchronization of \eqref{Ku-iner-net}, in this paper we   assume that the induced undirected graph $\mathcal G$ is {\em connected}. The following result, which connects the total deviations and the partial deviations along the edges in a connected graph, will be useful in the energy estimate. For its proof, we refer to \cite{C-L-H-X-Y}.
\begin{lemma}\label{Equivalence.comm}
Suppose that the graph $\mathcal G= (\mathcal V, \mathcal W, A)$ is connected and let $\theta_i$ be the phase of the Kuramoto oscillator located at the vertex $i$. Then, there there exists a positive constant $L_*$ such that
\[ \displaystyle L_* \sum_{l,k = 1}^{N} |\theta_l - \theta_k|^2 \leq  \sum_{(l, k) \in \mathcal W}
|\theta_l - \theta_k|^2 \leq \sum_{l,k =1}^{N}
|\theta_l - \theta_k|^2, \] where the positive constant $L_*$ is given by
\begin{equation}\label{L}
L_* := \frac{1}{1+d(\mathcal G)|\mathcal W^c|} \quad \mbox{with} \quad d(\mathcal G) := \max_{1\leq i,j \leq N}d_{ij}.
\end{equation}
Here $\mathcal W^c$ is the complement of edge set $\mathcal W$ in $\mathcal V\times \mathcal V$ and $|\mathcal W^c|$ denotes its cardinality.
\end{lemma}
%
\begin{remark}
$L_*$ has a strictly positive lower bound as
\[
L_* = \frac{1}{1+ d(\mathcal G)|\mathcal W^c|}  \geq \frac{1}{1+ d(\mathcal G)N^{2}}.
\]
\end{remark}

 \subsection{Main result}
 Based on Subsections \ref{subsecmodels} and \ref{s3.1}, our model for network-reduced lossless power grids can be restated as \eqref{Ku-iner-net} together with the restriction \eqref{eqsum0}, i.e.,
 \begin{align}\begin{aligned}\label{Ku-iner-net2}
&m_i\ddot{\theta}_i+d_i\dot{\theta}_{i}  = \Omega_{i} + \sum_{j=1}^{N} a_{ij} \sin(\theta_{j} - \theta_{i}),\,\, i=1,2,\dots,N, \\ &\sum_{i=1}^N \Omega_{i}=0,\qquad  a_{ij}=a_{ji}.
\end{aligned}\end{align}
  In this subsection, we present the  main
 result in this paper.  We begin by setting several extremal parameters:
\begin{align*}\begin{aligned} 
a_u :=\max\left\{a_{ij}: (j,i)\in \mathcal W \right\},& \quad a_{\ell}:=\min\left\{a_{ij}: (j,i)\in \mathcal W \right\}, \\
d_u := \max_{1 \leq i\leq N}d_i, \quad   d_\ell := \min_{1 \leq i\leq N}d_i,& \quad m_u := \max_{1 \leq i\leq N}m_i, \quad m_\ell := \min_{1 \leq i\leq N}m_i.
 \end{aligned}\end{align*}
We also set fluctuations of parameters:
\begin{align*}\begin{aligned}
\hat d_i:=d_i-\frac1N\sum_{i=1}^N d_i,& \quad  \hat D=diag(\hat d_1, \hat d_2,\dots, \hat d_N),\\
\hat m_i:=m_i-\frac1N\sum_{i=1}^N m_i,&  \quad \hat M=diag(\hat m_1, \hat m_2,\dots, \hat m_N).
 \end{aligned}\end{align*}
Using those notations, we introduce our main assumptions on the parameters and initial configurations below.
\begin{itemize}
\item[${\bf (H1)}$] The underlying graph $\mathcal G$ is connected.
\item[${\bf (H2)}$] Let $D_0 \in (0,\pi)$ be given. The parameters  satisfy
 \begin{equation}\label{assume}
a_u^2N^2(2m_u+\lambda)<d_\ell^2(2\mathcal R_0a_\ell L_*N-\lambda),
 \end{equation}
where $\mathcal{R}_0 := \frac{\sin D_0}{D_0}$, $L_*$ is given in Lemma \ref{Equivalence.comm}, and
\begin{equation}\label{lambda}
\lambda:=  \frac{\sqrt{tr({\hat D}^2)}}{\sqrt{N}} + \frac{2\sqrt{tr({\hat M}^2)}}{\sqrt{N}}.
\end{equation}
\item[${\bf (H3)}$] For some $\displaystyle \varepsilon\in \left(\frac{a_{u}^{2}N^2}{d_\ell\big(2\mathcal{R}_0a_\ell L_*N-\lambda\big)
},\,\, \frac{d_\ell}{2m_u+\lambda}\right)$,  the   parameters  and initial data satisfy
\begin{equation}\label{assumpB}
\max \lt\{ \sqrt{\wt\me(0)}, \frac{2\sqrt{2}C_1\max\{\e, 1\}\|\Omega\|}{\wt C_\ell \sqrt{C_0}}\rt\} <  \frac{\sqrt{C_0}}{2}D_0,
\end{equation}
where
\begin{align*}
 C_0 &:= \min \lt \{\frac{m_\ell}{2}, \e d_\ell \lt(1 - 2\varepsilon\frac{m_u}{d_\ell} \rt) \rt \},  \quad  C_1 :=\max \lt \{\frac{3m_u}{2}, \e d_u \lt(1 + 2\varepsilon\frac{m_u}{d_\ell} \rt) \rt\}, \cr
 \wt C_\ell &:= \min\lt\{d_\ell - 2\e m_u, 2\e \mathcal{R}_0 a_\ell L_*N - \frac{a_u^2 N^2}{d_\ell} \rt\} - \e\lambda,
\end{align*}
and
\[
\wt{\me}(0) :=\e\sum_{i=1}^N d_i (\theta_i(0)-\theta_c(0))^2+ 2\e\sum_{i=1}^N m_i (\theta_i(0) - \theta_c(0))\omega_i(0)+ \sum_{i=1}^N m_i \omega_i^2(0)
\]
with
\[
\theta_c(0) := \frac1N \sum_{i=1}^N \theta_i(0).
\]
\end{itemize}
Then we are now in a position to state our main theorem in this paper.
\begin{theorem}\label{thm4}
Suppose that the hypotheses ${\bf (H1)}$-${\bf (H3)}$ hold. Then  the global solution $\theta(t)$ to the system \eqref{Ku-iner-net2} asymptotically exhibits phase-locked states.
\end{theorem}

\subsection{Discussions}

  {We would like to explain about the accessibility of the assumptions ${\bf (H1)}$-${\bf (H3)}$.  The assumption ${\bf (H1)}$  guarantees the positivity of the constant $L_*$ appeared in \eqref{L}, and then ${\bf (H2)}$  can hold true, for example, when the inertia is small and the variances of inertia and damping are also  small.  Now,  the assumption ${\bf (H2)}$  
  ensures that the interval for admissible  $\varepsilon$ is nonempty, which further  guarantee that  $\wt C_\ell>0$.  Finally, the condition \eqref{assumpB} can be fulfilled  when the size of initial data  (in terms of the initial energy $\wt\me(0)$) and the size of (micro) natural frequencies $\|\Omega\|$ are small.}

  {By the definition of the perturbed matrices $\hat D$ and $\hat M$, we find   $\lambda = 0$ if $d_i = d$ and $m_i = m$ for all $1 \leq i \leq N$.  Moreover, in the case of uniform  inertia and damping, our assumptions ${\bf (H1)}$- ${\bf (H3)}$   become the ones in \cite{C-L-H-X-Y}.  
 }

  {We acknowledge that our estimate is conservative in the sense that the conditions are   sufficient but not necessary. In spite of that, Theorem \ref{thm4} gives  {\em explicit} formulas to   guarantee that a  given state  must be in the region of attraction for    synchronous states of a grid system by verifying that it meets the framework in ${\bf (H1)}$-${\bf (H3)}$,  which is easy to operated since only algebraic operations are involved.  We could also observe
some interesting points from the statement in Theorem \ref{thm4}. We notice that the parametric condition \eqref{assume} becomes more flexible  when we increase the constant $L_*$ and/or decrease the constant $\lambda$. Recalling \eqref{L} we see that if one decreases the diameter of the graph or increase the number of arcs, then the value of $L_*$ becomes larger and   the parametric condition is relaxed. On the other hand, by \eqref{lambda}, the constant $\lambda$ depends on the fluctuations of the nonuniform parameters $d_i$ and $m_i$; thus,   the parametric fluctuations hinder  the synchronization. These two observations are consistent with our intuition and  give  some qualitative understanding for the synchronizability versus the system parameters.}

  {Compared to   \cite{D-B-1, L-X-Y1}, the advantage of our main result lies in  at least two aspects. First,   we study the general systems in which  the  dampings can be inhomogeneous, i.e., the ratio of damping over inertia can be different between generators.  In  comparison, the analysis in \cite{L-X-Y1}  is  limited to the case of homogeneous dampings. 
Second, we extend the network topology to the most general case, i.e., the underlying graph can be arbitrary except the fundamental restriction that the network should be connected. Here,  the connectedness is indeed necessary  for synchronization;  otherwise, the oscillators in different components  cannot be expected to synchronize. In this sense, our assumption on the connectivity is most general.   In contrast,   the main result  in \cite{D-B-1} impliedly assumes that the underlying network is all-to-all interacted, i.e., each pair of nodes are connected to each other directly;   in \cite{L-X-Y1}, a basic hypothesis is that  the underlying graph should have a diameter less than or equal to 2.}

%
%

\section{Proof of main result: convergence to   phase-locked states}
\setcounter{equation}{0}
In this section, we give the proof of the main result, Theorem \ref{thm4}. Our main strategy can be summarized as follows. In Subsection \ref{sec_gradform}, we  present a gradient formulation of the system \eqref{Ku-iner-net2} and introduce some related theory. This theory tells   that the boundedness of trajectory implies its convergence.   Then, in order to show the boundedness,  we construct a virtual  energy functional in Subsection  \ref{sec_apriori}.  The energy functional  $ \wt \me(t)$    involves the fluctuation of phases around their averaged quantity \begin{equation}\label{thetacc}\theta_c(t) = \frac 1N\sum_{i=1}^N \theta_i(t).\end{equation}  In order to illustrate the reason to use such an energy, we begin with the  energy functional $\me(t)$ which was introduced in \cite{C-L-H-X-Y}.     In Subsection \ref{sec_main}, we combine the above   estimate  and  theory  to derive  the convergence to phase-locked states  for   the power networks \eqref{Ku-iner-net2}.

\subsection{A gradient-like flow formulation}\label{sec_gradform} In this part we present a new formulation of the system \eqref{Ku-iner-net} as a second-order gradient-like system in the case of symmetric capacity, i.e., $a_{ij} = a_{ji}$ for all $i,j \in \{1, 2, \dots,N\}$.  For the classic Kuramoto model, the potential function in the gradient flow   was first introduced in \cite{V-W}, which  can be extended to the  Kuramoto model with symmetric interactions.
 The following result was presented in \cite{C-L-H-X-Y}; we sketch the proof here for the reader.
 \begin{lemma}\label{lemgradform}
The system \eqref{Ku-iner-net} is a second-order gradient-like system with a real analytical potential $f$, i.e.,
\begin{equation}\label{gradsyst} M {\ddot \theta} + D{\dot \theta} = \nabla f(\theta), \end{equation}
if and only if  the adjacency matrix $A= (a_{ij})$ is symmetric.
\end{lemma}
\begin{proof} (i) Suppose that the   matrix $A$ is symmetric, i.e.,
$a_{ij}=a_{ji}.$
We define   $f: \mathbb{R}^N\to \mathbb{R}$ as
\begin{equation}\label{eqpot}
f(\theta) :=\sum_{k=1}^N\Omega_k \theta_k+\frac{1}{2}\sum_{k,l=1}^N a_{kl}\cos(\theta_k-\theta_l).
\end{equation}
It is clearly  analytic in $\theta$, and 
system \eqref{Ku-iner-net} is a second-order gradient-like system \eqref{gradsyst} with the potential $f$ defined in \eqref{eqpot}.

(ii) We now assume that the system \eqref{Ku-iner-net} is a gradient system with an analytic potential  $f$, i.e.,
\[ \frac{\partial f(\theta)}{\partial {\theta_i}} =\Omega_{i}+   \sum_{j=1}^{N} a_{ij} \sin(\theta_{j} - \theta_{i}),\quad i=1,2,\dots,N. \]
Then the potential $f$ must satisfy
$\displaystyle \frac{\partial^2 f }{\partial{\theta_k} \partial\theta_l}=  \frac{\partial^2 f }{\partial{\theta_l} \partial\theta_k}$ for $l\neq k.$
This concludes $a_{lk}=a_{kl}$ for all $l,k \in \{1,2,\dots,N\}$. 
\end{proof}

We next present a convergence result for the  second-order gradient-like system on $\mathbb R^N$:
\begin{equation}\label{mrgradient}
M\ddot\theta+D\dot\theta =\nabla  f(\theta), \quad \theta\in \mathbb{R}^N, \quad t\geq 0.
\end{equation}
Note  that the set of equilibria $\mathcal{S}$ coincides with the set of critical points of the potential $f$:
\[ {\mathcal S} := \{ \theta \in \bbr^N:~\nabla f(\theta) = 0 \}. \]
Based on the celebrated theory of {\L}ojasiewicz \cite{Lo1}, a convergence result of the gradient-like system with uniform inertia was established in \cite{H-J};  as a slight extension the following result was given in \cite{L-X-Y1}. 
\begin{lemma}\label{convergencethm2} \cite{L-X-Y1}
Assume that $f$ is analytic and let $\theta=\theta(t)$ be a global solution of \eqref{mrgradient}.
 If $\theta(\cdot)\in W^{1,\infty}(\mathbb R^+,\mathbb R^N)$, i.e., $\theta(\cdot) \in  L^\infty(\mathbb R^+,\mathbb R^N)$ and $\dot\theta(\cdot)\in  L^\infty(\mathbb R^+,\mathbb R^N)$, then there exists  an equilibrium   $\theta_e\in  \mathcal S$  such that
\[\lim_{t\to+\infty}\left\{\|\dot\theta(t)\|+\|\theta(t)-\theta_e\|\right\}=0.\]
\end{lemma}
%

 {Before we proceed, we first clarify that the Kuramoto oscillators are treated, in this paper, as a dynamic system on the whole space $\mathbb R^N$. Indeed, one can consider it as a system on the $N$-torus $\mathbb S^1\times \dots\times\mathbb S^1$ since the coupling function $\sin(\cdot)$ is $2\pi$-periodic. However, in order to apply the {\L}ojasiewicz's theory, we should treat the system \eqref{Ku-iner-net2} as a system on $\mathbb R^N$.  For more details  on {\L}ojasiewicz's theory  and applications, please refer to \cite{C-L-H-X-Y,H-J,L-X-Y,L-X-Y1}.}

  Then,  as  a direct application of Lemma \ref{convergencethm2}, we obtain {\it a priori} result on the complete frequency synchronization for \eqref{Ku-iner-net2}.
\begin{proposition} \label{convergencethm}
Let $\theta= \theta(t)$ be a solution to \eqref{Ku-iner-net2} in $W^{1, \infty}(\bbr^+, \bbr^N)$.
Then there exists  $\theta^\infty \in {\mathcal S}$ such that
$\lim_{t\to \infty}\{\|\dot\theta(t) \|+\|\theta(t)-\theta^\infty\|\}=0.$
\end{proposition}
The following lemma declares that $\dot\theta(\cdot) $ is  in $L^\infty(\mathbb R^+,\mathbb R^N)$  once $\theta(t)$ is  a solution of  the system \eqref{Ku-iner-net2}.
\begin{lemma}\label{lemmafreqbdd}
Let $\theta = \theta(t)$ be a solution to   \eqref{Ku-iner-net2}. Then   ${\dot \theta}(\cdot)\in    L^\infty(\mathbb R^+,\mathbb R^N)$.
\end{lemma}
\begin{proof}It follows from \eqref{Ku-iner-net2} that $\omega_i$ satisfies
\[
m_i \dot{\omega}_i + d_i\omega_i = \Omega_i +  \sum_{j=1}^N a_{ij}
\sin(\theta_j - \theta_i) \leq |\Omega_i| + \sum_{j=1}^N a_{ij}.
\]
 Note that $\omega_i$ is an analytic function of $t$. This implies that the zero-set $\{t:\omega_i(t) = 0 \}$ is countable and finite in any finite time-interval, i.e.,  $|\omega_i(t)|$ is piecewise differentiable and continuous. We multiply the above relation by $\mbox{sgn}(\omega_i)$ and divide it by $m_i > 0$ to get
\[ \frac{d |\omega_i|}{dt} + \frac{d_i}{m_i}|\omega_i| \leq \frac{1}{m_i} \lt( |\Omega_i| + \sum_{j=1}^N a_{ij} \rt), \quad \mbox{a.e. $t \geq 0$}. \]
We now use Gronwall inequality and continuity of $|\omega_i|$ to obtain that for all $t>0$,
\begin{align*}\begin{aligned}
|\omega_i(t)| \leq  |\omega_i(0)| e^{-\frac{d_i}{m_i}t} + \frac{1}{d_i}\lt( |\Omega_i| + \sum_{j=1}^N a_{ij} \rt)\lt(1 - e^{-\frac{d_i}{m_i}t}\rt)
\leq  |\omega_i(0)| +  \frac{1}{d_i}\lt( |\Omega_i| + \sum_{j=1}^N a_{ij} \rt),
\end{aligned}\end{align*}
due to $a_{ij} \geq 0$. This concludes the   boundedness of $\omega(t)=\dot\theta(t)$ as a function in time.
\end{proof}
 {\begin{remark}\label{remark1}
By Proposition \ref{convergencethm} and Lemma \ref{lemmafreqbdd}, to prove that the phase-locked states emerges in the system \eqref{Ku-iner-net2}, it suffices to  show  $\theta(\cdot)\in L^\infty(\mathbb R^+,\mathbb R^N)$, i.e., the  trajectory of phase  is   bounded.
\end{remark}
\begin{remark}\label{remark2}
Considering the system of  coupled oscillators \eqref{Ku-iner-net}, with general natural frequencies with $\sum_{i=1}^N \Omega_i\neq0$, we cannot expect the trajectory $\theta(t)=(\theta_1(t),\dots,\theta_N(t))$ be bounded in $\mathbb R^N$, since the right hand side of  \eqref{Ku-iner-net} sums to $\sum_{i=1}^N \Omega_i\neq0$.  This is why we apply the macro-micro decomposition and  define the micro-variables in Section \ref{s3.1},  which allows us to assume  without loss of any generality  that $\sum_{i=1}^N\Omega_i=0$ and reduces to the model \eqref{Ku-iner-net2}. In the next subsection, this restriction will be crucially used.
\end{remark} }

\subsection{Construction of the energy functional $\wt\me$}\label{sec_apriori}
Inspired by \cite{C-L-H-X-Y}, we first introduce a temporal energy functional $\me$:
for $\varepsilon > 0$,
\begin{align}\begin{aligned}  \label{energy-1}
\mathcal{E} [\theta, \omega]&:= \varepsilon  \langle D\theta, \theta\rangle  + 2 \varepsilon\langle M\theta, \omega \rangle +
\langle M\omega, \omega\rangle\\& =\varepsilon \sum_{i=1}^N d_i \theta_i^2+ 2\varepsilon \sum_{i=1}^N m_i \theta_i\omega_i+ \sum_{i=1}^N m_i \omega_i^2.
\end{aligned}\end{align}
Here, the notation $\langle \cdot, \cdot\rangle$ represents the standard inner product in $\mathbb R^N$.
Then we easily find the equivalence relation between $\mathcal{E}[\theta,\omega]$ and $\|\theta\|^2 + \|\omega\|^2$.
\begin{lemma}\label{equivalence}
Let $\varepsilon\in \lt(0,\, \frac{d_\ell}{2m_u} \rt)$. Then we have the following relation: 
\[
C_0 (\|\theta\|^2+\|\omega\|^2) \leq {\mathcal
E}[\theta, \omega] \leq C_1 (\|\theta\|^2+\|\omega\|^2),  \quad \forall\,\theta, \omega\in \mathbb R^N, 
\]
where $C_0$ and $C_1$ are positive constants (independent of $(\theta, \omega)$)  given by
\begin{align*}
\begin{aligned}
C_0 := \min \lt \{\frac{m_\ell}{2}, \e d_\ell\lt(1 - 2\varepsilon\frac{m_u}{d_\ell} \rt) \rt \} \quad \mbox{and} \quad C_1 :=\max \lt \{\frac{3m_u}{2}, \e d_u \lt(1 + 2\varepsilon\frac{m_u}{d_\ell} \rt) \rt\}.
\end{aligned}
\end{align*}
\end{lemma}
\begin{proof} In \eqref{energy-1}, the cross term $  \theta_i \omega_i  $ can be estimated by  Young's inequality:
\[
|  \theta_i  \omega_i |\leq \varepsilon\theta_i^2  + \frac{\omega_i^2}{4\varepsilon}. \]
Then, we have
\[
2\varepsilon m_i |\theta_i \omega_i  |\leq 2\varepsilon^2m_i\theta_i^2  + \frac{m_i}{2}\omega_i^2 \leq 2\e^2 \frac{m_u}{d_\ell} d_i\theta_i^2 +  \frac{m_i}{2}\omega_i^2 ,
\]
and hence
\[
\sum_{i=1}^N\frac{m_i}{2}\omega_i^2+\varepsilon d_\ell \lt(1 - 2\varepsilon\frac{m_u}{d_\ell} \rt)\sum_{i=1}^N \theta_i^2\leq{\mathcal
E}[\theta, \omega] \leq
\sum_{i=1}^N\frac{3m_i}{2}\omega_i^2+\varepsilon d_u\lt(1 + 2\varepsilon\frac{m_u}{d_\ell} \rt)\sum_{i=1}^N \theta_i^2.
\]
This gives the desired result.
\end{proof}

\begin{lemma}\label{lemma-energy}Let $D_0\in (0, \pi)$ and  suppose that the phase configuration $\{\theta_i\}_{i=1}^N$ satisfies
\[   \max_{1 \leq i, j \leq N} |\theta_i - \theta_j | \leq D_0. \]
Then the following estimate holds.
\begin{eqnarray*}
&&(i)\,\,\,\,\,\,a_{u}\sum_{(i,j)\in
\mathcal W}\Big|\sin(\theta_{j}-\theta_i)(\omega_j-\omega_i)\Big| \leq \frac{a_u^2 N^2}{d_\ell}\|\theta - \theta_c\|^2 + d_\ell\|\omega\|^2. \cr
&& (ii)\,\, \sum_{(i,j)\in \mathcal W} a_{ij} \sin(\theta_j-\theta_i)(\theta_j-\theta_i) \geq 2\mathcal{R}_0 a_{\ell} L_* N\|\theta-\theta_c\|^2,
\end{eqnarray*}
where $\mathcal{R}_0$ is given by $\mathcal{R}_0 := \frac{\sin D_0}{D_0}$, and   the vector $\theta-\theta_c$ is understood as $\theta-\theta_c:=(\theta_1,\dots,\theta_N)-(\theta_c,\dots,\theta_c)$ with $\theta_c$ given in \eqref{thetacc}.
\end{lemma}
\begin{proof}
(i) We use  $|\sin(\theta_j-\theta_i)| \leq |\theta_j-\theta_i|$ and Young's inequality to obtain
$$\begin{aligned}
a_{u} \sum_{(i,j) \in \mathcal W}\Big|\sin(\theta_j-\theta_i)(\omega_j-\omega_i)\Big| &\leq \frac{a_{u}^2 N}{2d_\ell} \sum_{(i,j)\in \mathcal W} |\theta_j-\theta_i|^2+ \frac{d_\ell}{2N} \sum_{(i,j)\in
\mathcal W} |\omega_j-\omega_i|^2\cr
&\leq \frac{a_u^2 N}{2d_\ell}\sum_{1 \leq i,j \leq N}|\theta_i - \theta_j|^2 + d_\ell\|\omega\|^2\cr
&= \frac{a_u^2 N^2}{d_\ell}\|\theta - \theta_c\|^2 + d_\ell\|\omega\|^2,
\end{aligned}$$
where we used the relations that
\[\sum_{1 \leq i,j \leq N}|\theta_i - \theta_j|^2=2N\|\theta - \theta_c\|^2,\]
and \[
\sum_{(i,j)\in
\mathcal W} |\omega_j-\omega_i|^2 \leq \sum_{1 \leq i,j \leq N}|\omega_i - \omega_j|^2 =2N\|\omega - \omega_c\|^2 \leq  2N\|\omega\|^2.
\]
(ii)~ It follows from the assumption
\[
\max_{1 \leq i,j \leq N}|\theta_j-\theta_i|\leq D_{0} < \pi,
\]
and   the simple relation
\[
x\sin x\geq \mathcal{R}_0 x^2 \quad \mbox{for} \quad x \in [-D_0,D_0],
\]
that
$$\begin{aligned}
\sum_{(i,j) \in \mathcal W}a_{ij}\sin(\theta_j-\theta_i)(\theta_j-\theta_i) &\geq \mathcal{R}_0\sum_{(i,j)\in
\mathcal W}a_{ij}|\theta_j-\theta_i|^2\cr
&\geq \mathcal{R}_0 a_{\ell} L_* \sum_{1\leq i,j\leq N}|\theta_j-\theta_i|^2\cr &=2\mathcal{R}_0 a_{\ell} L_* N\|\theta-\theta_c\|^2.
\end{aligned}$$
Here $L_*$ is the positive constant explicitly defined in Lemma \ref{Equivalence.comm}.
\end{proof}
Recall that the system \eqref{Ku-iner-net2} can be rewritten as
\begin{align}
\begin{aligned} \label{first-KMI}
{\dot \theta}_i &= \omega_i, \quad i=1, 2, \dots, N, \quad t > 0, \\
{\dot \omega}_i &= \frac{1}{m_i} \left[ -d_i\omega_i +   \Omega_i +  \sum_{j=1}^{N}
a_{ij}\sin(\theta_j - \theta_i) \right],\\ &\sum_{i=1}^N \Omega_{i}=0,\qquad  a_{ij}=a_{ji}.
\end{aligned}
\end{align}
Next, we present quantitative estimates of the interaction force term.
For notational simplicity, we denote
\[
{\mathcal E}(t) := {\mathcal E}[\theta(t), \omega(t)], \quad t \geq 0.
\]
where $(\theta(t),\omega(t))$ is the solution to the system \eqref{Ku-iner-net2} or \eqref{first-KMI}.
\begin{proposition}\label{EnergyEstimateLemma} Let $D_0\in (0,\pi)$ and $\{\theta_i\}_{i=1}^N$ be any smooth solution to the system \eqref{Ku-iner-net2}.
Suppose that
\[
 a_u^2N m_u <  d_\ell^2 \mathcal R_0 a_\ell  L_*   \quad \mbox{and} \quad \max_{t\in [0,T_0]} \max_{1 \leq i, j \leq N} |\theta_i(t) - \theta_j(t) | \leq D_0. \]
 for some $T_0 > 0$.
Then, for any  $\varepsilon$ satisfying
\bq\label{condi_e}
\frac{a_{u}^{2}N}{2d_\ell
\mathcal{R}_0 a_\ell L_*}<\varepsilon<\frac{d_\ell}{2m_u},
\eq
we have
\begin{equation}\label{energy-0}
\frac{d}{dt}\me(t) + C_\ell \md(t) \leq 2\max\{\e, 1\}\| \Omega\|\lt( \|\theta - \theta_c\| + \|\omega\|\rt), \quad \mbox{for} \,\,\,t\in[0, T_0],
\end{equation}
where   $\md(t) := \md[\theta(t),\omega(t)]$   and $C_\ell$  are defined by
\[
\md[\theta,\omega]:= \|\omega\|^2 + \|\theta - \theta_c\|^2 \quad \mbox{and} \quad C_\ell:= \min\lt\{d_\ell - 2\e m_u, 2\e \mathcal{R}_0 a_\ell L_*N - \frac{a_u^2 N^2}{d_\ell} \rt\}.
\]
\end{proposition}
\begin{proof}  The proof is divided into three steps.

$\bullet$ {\bf Step A.-} We multiply $2\omega_i$ on both sides of the second equation in
$\eqref{first-KMI}_2$, sum it over $i$, and then use the symmetry of $a_{ij}$ and  Lemma \ref{lemma-energy} to obtain
\begin{align*}
\begin{aligned}
 \frac{d}{dt}  \sum_{i=1}^N m_i\omega_i^2 &= -2 \sum_{i=1}^N d_i\omega_i^2   + 2 \sum_{i=1}^N \Omega_i \omega_i + 2 \sum_{i,j=1}^{N}a_{ij}\sin
(\theta_{j}-\theta_i)\omega_i \\
&=-2 \sum_{i=1}^N d_i\omega_i^2   + 2 \sum_{i=1}^N \Omega_i \omega_i -  \sum_{i,j=1}^{N}a_{ij}\sin
(\theta_{j}-\theta_i)(\omega_j- \omega_i)\\
&\leq -2 \sum_{i=1}^N d_i\omega_i^2  + 2 \sum_{i=1}^N \Omega_i \omega_i  +   a_{u}\sum_{(i,j)\in
\mathcal W}\Big|\sin(\theta_{j}-\theta_i)(\omega_j-\omega_i)\Big| \\
&\leq -2 \sum_{i=1}^N d_i\omega_i^2   +  2 \|\Omega\| \|\omega\| +\frac{a_u^2 N^2}{d_\ell}\|\theta - \theta_c\|^2 + d_\ell\|\omega\|^2.
\end{aligned}
\end{align*}
This yields
\bq\label{sumoveri.comm}
\frac{d}{dt}\lal M\omega,\omega\ral \leq-d_\ell \|\omega\|^2 + 2\|\Omega\|\|\omega\| + \frac{a_u^2 N^2}{d_\ell}\|\theta - \theta_c\|^2.
\eq

$\bullet$ {\bf Step B.-} We now multiply $2\theta_i$ on both sides of $\eqref{first-KMI}_2$ to obtain
\begin{align*}
\begin{aligned}
2m_i \lt( \frac{d \omega_i}{dt}   \rt) \theta_i = - d_i\frac{d}{dt} \theta_i^2 +  2 \Omega_i \theta_i
+ 2\sum_{j=1}^{N}a_{ij}\sin(\theta_j-\theta_i)\theta_i.
\end{aligned}
\end{align*}
Summing the above equality over $i$ and using the symmetry of $a_{ij}$ and Lemma \ref{lemma-energy}, we find
\begin{align}\label{est_wt}
\begin{aligned}
2 \sum_{i=1}^{N} m_i\lt( \frac{d \omega_i}{dt}\rt) \theta_i &= -\frac{d}{dt}\sum_{i=1}^N d_i\theta_i^2  +
2 \sum_{i=1}^N \Omega_i \theta_i + 2\sum_{(j,i)\in
\mathcal W}a_{ij}\sin(\theta_j-\theta_i)\theta_i\\ &= -\frac{d}{dt}\sum_{i=1}^N d_i\theta_i^2+ 2 \sum_{i=1}^N
\Omega_i \theta_i -  \sum_{(j,i)\in \mathcal W}a_{ij}\sin(\theta_j-\theta_i)(\theta_j-\theta_i) \\&= -\frac{d}{dt}\sum_{i=1}^N d_i\theta_i^2+ 2 \sum_{i = 1}^N\Omega_i (\theta_i - \theta_c)  -  \sum_{(j,i)\in \mathcal W}a_{ij}\sin(\theta_j-\theta_i)(\theta_j-\theta_i) \\
&\leq -\frac{d}{dt}\sum_{i=1}^N d_i\theta_i^2 + 2 \|\Omega\| \|\theta - \theta_c \|   - 2\mathcal{R}_0 a_{\ell} L_* N\|\theta-\theta_c\|^2,
\end{aligned}
\end{align}
where  we used
 the restriction that
\[
\sum_{i = 1}^N \Omega_i = 0.
\]
On the other hand, the term in the left hand side of relation \eqref{est_wt} can be rewritten as
\bq\label{est_wt2}
m_i\frac{d\omega_i}{dt} \theta_i=m_i \frac{d}{dt} (\omega_i\theta_i)- m_i \omega_i ^2.
\eq
Combining \eqref{est_wt} and \eqref{est_wt2}, we obtain
\begin{equation*}
\frac{d}{dt}\left( 2   \sum_{i=1}^N m_i \omega_i \theta_i +\sum_{i=1}^N d_i\theta_i^2 \right) + 2\mathcal{R}_0 a_{\ell} L_* N\|\theta-\theta_c\|^2  \leq 2\| \Omega\| \|\theta - \theta_c\| + 2 \sum_{i=1}^N m_i\omega_i^2.
\end{equation*}
Finally, we use the fact
\[
\sum_{i=1}^N m_i \omega_i^2 \leq m_u\|\omega\|^2,
\]
to conclude
\begin{align}\label{New-1}
\begin{aligned}
\frac{d}{dt}\lt( \langle D\theta, \theta\rangle  + 2 \langle M\theta, \omega \rangle\rt) + 2\mathcal{R}_0 a_{\ell} L_* N\|\theta-\theta_c\|^2 \leq 2\| \Omega\| \|\theta - \theta_c\| + 2m_u\|\omega\|^2.
\end{aligned}
\end{align}

$\bullet$ {\bf Step C.-} Taking \eqref{sumoveri.comm} + $\varepsilon \times$ \eqref{New-1} yields
$$\begin{aligned}
\frac{d}{dt} \mathcal{E}(t) + (d_\ell - 2\e m_u)&\|\omega\|^2 + \lt( 2\e \mathcal{R}_0 a_\ell L_*N - \frac{a_u^2 N^2}{d_\ell}\rt) \|\theta - \theta_c\|^2 \cr
&  \leq 2\max\{\e, 1\}\| \Omega\|\Big( \|\theta - \theta_c\| + \|\omega\|\Big).
\end{aligned}$$
Then it follows from the condition on $\e > 0$ in \eqref{condi_e} that
\[
\frac{d}{dt}\me(t) + C_\ell \md(t) \leq 2\max\{\e, 1\}\| \Omega\|\Big( \|\theta - \theta_c\| + \|\omega\|\Big), \quad \mbox{for}\,\,\,t\in [0,T_0].
\]
This is the desired inequality and the proof is completed.
\end{proof}
It follows from the definition of $\md[\theta,\omega]$ and Lemma \ref{equivalence} that
\[
\md[\theta,\omega] \leq \|\omega\|^2 +  \|\theta\|^2   \leq \frac{1}{C_0}\me[\theta,\omega],
\]
or equivalently,
\[
 C_0   \md[\theta,\omega] \leq \me[\theta,\omega].
\]
However, we can easily find that the functional $\me[\theta,\omega]$ is not bounded from above by the dissipation rate $\md[\theta,\omega]$. In the case  of uniform inertia and damping \cite{C-H-Y, C-L-H-X-Y},  
  applying a macro-micro decomposition if necessary, we can assume    $\theta_c(t)=0$   for all $t \geq 0$, which implies that along the flow \eqref{Ku-iner-net2} we have
\[
\frac{1}{2N}\sum_{1 \leq i,j \leq N}|\theta_i - \theta_j|^2 = \sum_{i=1}^N \theta_i^2, \quad \forall\, t>0.
\]
This immediately implies that $\me(t)$ is bounded from above by $\md(t)$ uniformly in time, and thus, they are equivalent.  Then we can derive a nice differential inequality on $\me(t)$ from \eqref{energy-0} which  enables us to obtain  the uniform boundedness of the temporal energy functional $\me(t)$   under suitable  initial configurations.  However, in the current case with non-uniform parameters, the average quantity $\theta_c(t)$ is not conserved. As a consequence, the dissipation $\md(t)$ does not provide a   damping effect for the energy functional $\me(t)$. In order to obtain a proper dissipation of energy, we introduce a modified   energy functional $\wt\me$ as:
\[
\wt{\me}[\theta,\omega] :=\e\sum_{i=1}^N d_i (\theta_i-\theta_c)^2+ 2\e\sum_{i=1}^N m_i (\theta_i - \theta_c)\omega_i+ \sum_{i=1}^N m_i \omega_i^2\quad \mbox{with} \quad \theta_c = \frac1N \sum_{i=1}^N \theta_i.
\]
In the lemma below, we provide some relations between $\me$ and $\wt\me$, and $\md$ and $\wt\me$.
\begin{lemma}\label{lemma_newenergy}
(1)   The functionals $\me$ and $\wt\me$ have the following relation:  \begin{align}\label{diff_et}
\begin{aligned}
\wt\me = \me - 2\e \theta_s \theta_c + \e \,tr(D)\theta_c^2 - 2\e \omega_s \theta_c \,,
\end{aligned}
\end{align} where $\theta_s$ and $\omega_s$ are given as in \eqref{vs}. \newline
(2)  The  functional $\wt\me$ and the dissipation rate $\md$ are equivalent:
 \bq\label{eqv_ed}
 C_0  \md[\theta,\omega] \leq \wt\me[\theta,\omega] \leq C_1 \md[\theta,\omega], \quad \forall\,\theta, \omega\in \mathbb R^N,
\eq  where $C_0$ and $C_1$ are positive constants given as in Lemma \ref{equivalence}.
 \end{lemma}
\begin{proof}
(1)  The relation between $\me$ and $\wt\me$ immediately  follows from the definition of $\wt\me$:
\begin{align*}
\begin{aligned}
\wt\me &= \me - 2\e\sum_{i=1}^N d_i \theta_i \theta_c + \e \sum_{i=1}^N d_i \theta_c^2 - 2\e \sum_{i=1}^N m_i \omega_i \theta_c\cr
&= \me - 2\e \theta_s \theta_c + \e \,tr(D)\theta_c^2 - 2\e \omega_s \theta_c\,.
\end{aligned}
\end{align*}
(2)  Replacing  the term $\theta$ by $\theta-\theta_c$ in Lemma \ref{equivalence} yields the desired estimate
\begin{equation*}\label{est_equiv}
C_0 (\|\theta - \theta_c\|^2+\|\omega\|^2) \leq \wt\me[\theta, \omega] \leq C_1 (\|\theta - \theta_c\|^2+\|\omega\|^2).
\end{equation*}
\end{proof}
We now present the time-evolution of the modified energy functional
\[
\wt\me(t):=\wt\me[\theta(t),\omega(t)].
\]
Before we proceed, we first mention an   conservation property which is   important  in the  upcoming  estimate.
   {\begin{lemma}\label{remarksum0-1}
The sum of weighted average   is conserved in time:
\begin{equation}\label{eqderive0}
\dot\theta_s+\dot\omega_s=0.
\end{equation}
\end{lemma}}
\begin{proof}
This immediately follows from  \eqref{Ku-iner-net2}. In particular, here we used the restriction $\sum_{i=1}^N \Omega_{i}=0$.
\end{proof}
\begin{proposition}\label{prop_final}
 Let $D_0\in (0,\pi)$ and $\{\theta_i\}_{i=1}^N$ be any smooth solution to the system \eqref{Ku-iner-net2}.
Suppose that
 \begin{equation*} 
 a_u^2N^2(2m_u+\lambda)<d_\ell^2(2\mathcal R_0a_\ell L_*N-\lambda) \quad \mbox{with} \quad \lambda=  \frac{\sqrt{tr({\hat D}^2)}}{\sqrt{N}} + \frac{2\sqrt{tr({\hat M}^2)}}{\sqrt{N}},
 \end{equation*}
and
\begin{equation}\label{assumD}
\max_{t\in [0,T_0]} \max_{1 \leq i, j \leq N} |\theta_i(t) - \theta_j(t) | \leq D_0,
\end{equation}
for some $T_0 > 0$.
Then, for any  $\varepsilon$ satisfying
\begin{equation*}
\frac{a_{u}^{2}N^2}{d_\ell(2\mathcal{R}_0a_\ell L_*N-\lambda)
}<\varepsilon<\frac{d_\ell}{2m_u+\lambda},
\end{equation*}
we have
\bq\label{ineq_main}
\frac{d}{dt}\wt\me(t) + \wt C_\ell \md(t) \leq \frac{2\sqrt{2}\max\{\e, 1\}\| \Omega\|}{\sqrt{C_0}} \sqrt{\wt\me(t)}, \quad \mbox{for}\,\,\,t\in [0, T_0],
\eq
where $\wt C_\ell$ is a positive constant given by
$
\wt C_\ell := C_\ell - \e\lambda.
$
Moreover, we have  \begin{equation}\label{ineqmain1}
\frac{d}{dt} \wt \me(t) + \frac{\wt C_\ell}{C_1} \wt\me(t) \leq \frac{2\sqrt{2}\max\{\e, 1\}\| \Omega\|}{\sqrt{C_0}} \sqrt{\wt\me(t)}, \quad \mbox{for}\,\,\,t\in [0, T_0].
\end{equation}
\end{proposition}
\begin{proof} It follows from  Proposition \ref{EnergyEstimateLemma} and \eqref{diff_et} in Lemma \ref{lemma_newenergy}   that $\wt\me$ satisfies
\[
\frac{d}{dt}\wt \me(t) + C_\ell \md(t) \leq  \underbrace{\frac{d}{dt} \lt( \e \,tr(D)\theta_c^2  - 2\e \theta_s \theta_c - 2\e \omega_s \theta_c\rt)}_{=:I} + \underbrace{2\max\{\e, 1\}\| \Omega\|\lt( \|\theta - \theta_c\| + \|\omega\|\rt)}_{=: J}.
\]
Using \eqref{eqderive0}, we rewrite $I$ as
$$\begin{aligned}
I &= 2\e\, tr (D)\theta_c \dot\theta_c  - 2\e \dot\theta_s \theta_c - 2\e \theta_s \dot\theta_c - 2\e \dot\omega_s \theta_c - 2\e \omega_s \dot\theta_c\cr
&= 2\e\, tr (D)\theta_c \dot\theta_c  - 2\e \theta_s \dot\theta_c - 2\e \omega_s \dot\theta_c \quad \lt(\because \dot\theta_s + \dot\omega_s = 0\rt)\cr
&= -2\e\dot\theta_c \sum_{i=1}^N d_i(\theta_i - \theta_c) -2\e\omega_s \dot\theta_c \quad \lt(\because \theta_s = \sum_{i=1}^N d_i(\theta_i - \theta_c) + tr(D)\theta_c\rt)\cr
&= -2\e\omega_c \sum_{i=1}^N d_i(\theta_i - \theta_c) -2\e\omega_s \omega_c \quad \lt( \because \dot\theta_c = \omega_c := \frac1N\sum_{i=1}^N \omega_i\rt).
\end{aligned}$$
Note that
\[
\sum_{i=1}^N d_i(\theta_i - \theta_c) = \sum_{i=1}^N \hat d_i(\theta_i - \theta_c) \quad \mbox{and} \quad \omega_s = \sum_{i=1}^N \hat m_i \omega_i + tr(M) w_c.
\]
This yields
$$\begin{aligned}
I  &= -2\e\omega_c \sum_{i=1}^N \hat d_i(\theta_i - \theta_c) -2\e\lt(\sum_{i=1}^N \hat m_i \omega_i + tr(M) w_c\rt) \omega_c \cr
&\leq -2\e\omega_c \sum_{i=1}^N \hat d_i(\theta_i - \theta_c) -2\e \omega_c \sum_{i=1}^N \hat m_i \omega_i.
\end{aligned}$$
On the other hand, we find
$$\begin{aligned}
&\lt|2\e\omega_c \sum_{i=1}^N \hat d_i(\theta_i - \theta_c) +2\e \omega_c \sum_{i=1}^N \hat m_i \omega_i\rt|\cr
&\quad =\lt| \frac{2\e}{N}\lt(\sum_{i=1}^N \omega_i \rt)\lt(\sum_{i=1}^N \hat d_i(\theta_i - \theta_c) \rt)+ \frac{2\e}{N}\lt(\sum_{i=1}^N \hat m_i \omega_i \rt)\lt(\sum_{i=1}^N \omega_i\rt)\rt|\cr
&\quad\leq \frac{2\e}{N}\sqrt{N}\|\omega\| \sqrt{tr({\hat D}^2)}\|\theta - \theta_c\| + \frac{2\e}{N}\sqrt{tr({\hat M}^2)} \|\omega\| \sqrt{N}\|\omega\|\cr
&\quad= \frac{2\e}{\sqrt{N}}\sqrt{tr ({\hat D}^2)}\|\omega\|\|\theta  - \theta_c\| + \frac{2\e \sqrt{tr({\hat M}^2)}}{\sqrt{N}}\|\omega\|^2\cr
&\quad\leq \e\lt(\frac{\sqrt{tr({\hat D}^2)}}{\sqrt{N}}\|\omega\|^2 +\frac{\sqrt{tr({\hat D}^2)}}{\sqrt{N}} \|\theta-\theta_c\|^2\rt)+ \frac{2\e\sqrt{tr({\hat M}^2)}}{\sqrt{N}} \|\omega\|^2 \cr
&\quad\leq \e\lt(\frac{\sqrt{tr({\hat D}^2)}}{\sqrt{N}} + \frac{2\sqrt{tr({\hat M}^2)}}{\sqrt{N}}\rt) \md[\theta,\omega].
\end{aligned}$$
Thus, we have $$I\leq \e\lt(\frac{\sqrt{tr({\hat D}^2)}}{\sqrt{N}} + \frac{2\sqrt{tr({\hat M}^2)}}{\sqrt{N}}\rt) \md[\theta,\omega]. $$
For the estimate of $J$, we obtain
$$\begin{aligned}
J=2\max\{\e, 1\}\| \Omega\|\lt( \|\theta - \theta_c\| + \|\omega\|\rt) &\leq 2\sqrt{2}\max\{\e, 1\}\| \Omega\|\sqrt{ \|\theta - \theta_c\|^2 + \|\omega\|^2}\cr
& \leq \frac{2\sqrt{2}\max\{\e, 1\}\| \Omega\|}{\sqrt{C_0}} \sqrt{\wt\me(t)},
\end{aligned}$$
where we used the  elementary relation $a + b \leq \sqrt{2}\sqrt{a^2 + b^2}$ for $a,b \geq 0$ and Lemma \ref{lemma_newenergy} (2).
We now combine the above estimates  for $I$ and $J$ to see that, for $t\in[0, T_0]$,
\[
\frac{d}{dt}\wt\me(t) + \lt(C_\ell - \e\lambda \rt)\md(t) \leq \frac{2\sqrt{2}\max\{\e, 1\}\| \Omega\|}{\sqrt{C_0}} \sqrt{\wt\me(t)}.
\]
This is the desired inequality \eqref{ineq_main}.
Finally,  the last inequality \eqref{ineqmain1} immediately follows from \eqref{eqv_ed} and \eqref{ineq_main}.
\end{proof}

%
%
\subsection{Proof of Theorem \ref{thm4}}\label{sec_main}
For the sake of notational simplicity, we set
\[
y(t) := \sqrt{\wt\me(t)} \quad t \geq 0.
\]
Define
 \[
\mathcal{T} := \lt\{ T \in \mathbb{R}_{+} : y(t)<\frac{\sqrt{C_0}}{2}D_0, \quad \forall\,
t \in [0,T) \rt\}, \quad   {T}^* := \sup \mathcal{T}.
\]
Note that by the assumption \eqref{assumpB},
\[
y(0) < \frac{\sqrt{C_0}}{2}D_0.
\]
  Due to the continuity of $y$,    there exists a positive constant $T>0$ such that $T\in \mathcal T$.
We now claim that \begin{equation}\label{claim3}T^*=\infty.\end{equation}
Suppose the opposite, i.e., $T^*$ is finite. Then, we should have
\begin{equation}
\label{eitheror}y(T^*)=\frac{\sqrt{C_0}}{2}D_0.
\end{equation}
Note that on the interval $[0,T^*)$,    we can derive that
$$\begin{aligned}
\max_{1 \leq i,j \leq N}|\theta_i(t)-\theta_j(t)|^2 &\leq 4\max_{1 \leq i \leq N}|\theta_i(t) - \theta_c(t)|^2 \leq 4\sum_{i=1}^N|\theta_i(t) - \theta_c(t)|^2\cr
&  \leq 4\md(t)\leq \frac{4}{C_0}\wt\me(t)\cr
& \leq \frac{4}{C_0}\lt( \frac{\sqrt{C_0}}{2}D_0\rt)^2 = D_0^2,
\end{aligned}$$
which means that the condition \eqref{assumD} is fulfilled, and then Proposition \ref{prop_final}  can be applied. 
 By \eqref{ineqmain1} we have
\begin{equation}\label{EnergyMain}
\frac{d y}{dt} \leq \frac{\sqrt{2}\max\{\e, 1\}\| \Omega\|}{\sqrt{C_0}} - \frac{\wt C_\ell}{2C_1} y, \quad \mbox{for} \quad t \in [0,T^*].
\end{equation}
Note that the solution $y(t)$ to the system \eqref{EnergyMain} satisfies
\[
y(T^*) \leq \max \lt\{ y(0), \frac{2\sqrt{2}C_1\max\{\e, 1\}\|\Omega\|}{\wt C_\ell \sqrt{C_0}}\rt\} <  \frac{\sqrt{C_0}}{2}D_0,
\]
where we used the assumption  \eqref{assumpB}. This   contradicts   \eqref{eitheror} and    the claim \eqref{claim3} is proved, i.e.,
 \[
\wt \me (t) <\frac{C_0}{4}D_0^2, \qquad \forall~ t\geq0.
\]
This implies that
\begin{equation}\label{diffbound}
\max_{1 \leq i,j \leq N}|\theta_i(t)-\theta_j(t)|^2\leq 4\md (t) \leq \frac{4}{C_0}\wt\me(t)<D_0^2, \qquad \forall~ t\geq 0.
\end{equation}
On the other hand, we  recall the relation \eqref{eqderive0} to get
\[
\omega_s(t) + \theta_s(t) = \omega_s(0) + \theta_s(0), \qquad \forall~t \geq 0.
\]
This means that
\[
|\theta_s(t)| \leq |\omega_s(t) + \theta_s(t)| + |\omega_s(t)| = |\omega_s(0) + \theta_s(0)| + |\omega_s(t)|,\quad  \forall \,t\geq0.
\]
We now use the fact that $\omega(\cdot) \in    L^\infty(\mathbb R^+,\mathbb R^N)$ 
in Lemma \ref{lemmafreqbdd} to deduce
\begin{equation}\label{weightsumbdd}
|\theta_s(t)| \leq K_0, \quad \forall \, t\geq0,
\end{equation}
for some positive constant $K_0$.  Combining the relations  \eqref{diffbound} and \eqref{weightsumbdd}, we see that the trajectory  $\theta(\cdot)$ is  bounded as a function in time $t$.
 Thus, we obtain  $\theta(\cdot)\in W^{1,\infty}(\mathbb R^+, \mathbb R^N)$ (see Remark \ref{remark1}), since   $\dot\theta(\cdot)$
is bounded by Lemma \ref{lemmafreqbdd}.  Finally, we apply Proposition \ref{convergencethm}
 to find that the system \eqref{Ku-iner-net2} asymptotically attains the phase-locked states.
 This completes the proof.
\begin{remark}
If, in addition,  $D_0\leq\pi/2$, then the emergent phase-locked state must be confined in an arc with length less than $\pi/2$. Thus, the result in \cite[Theorem 3.1]{L-X-Y1} holds.  Furthermore,  by appealing to the  approach  in \cite{L-X-Y1} (see the Step 2 in the proof of Theorem 2.1),  we can derive that the convergence to the phase-locked states is exponentially fast.
 \end{remark}

\begin{remark} In our approach, the function $\wt\me$ is not a physical energy, so it can be regarded as a virtual energy.
 This virtual  energy functional is different from that in \cite{C-L-H-X-Y} where the case of uniform inertia and damping was considered. Actually, the uniformity implies some nice property so that   the  mean value of phases can be assumed to be zero all the time. This played important roles in that analysis. In this work, this property is absent due to the non-uniform parameters; thus,  we construct  the new energy functional $\wt\me$ to overcome this difficulty.
\end{remark}

%
%
%
%
%
%
%

\section{Numerical simulations}
\setcounter{equation}{0}
  {In ${\bf (H2)}$ and ${\bf (H3)}$, the parameters $D_0$ and $\varepsilon$  are chosen from  some open intervals, then the estimated region of attraction is different upon  different choices. 
As we see in \eqref{diffbound}, the constant $D_0$ is actually the range of phases for the system.
In the statement of Theorem \ref{thm4}, it is pre-assigned in $(0,\pi)$ which needs to fit \eqref{assume}.  Its value affects the admissible range of $\e$, other constants and the right hand side of \eqref{assumpB}. On the other hand, the choice of $\varepsilon$ affects the energy functional $\tilde \me$ and other constants.  So, it would be interesting to investigate the region of attraction with different range of phases energy functional and different energy functional.   
In this section, we will do some simulations and illustrate the influence of $D_0$ and $\varepsilon$ on the estimated region, for a special setting. The conservativeness of our estimate is also illustrated.}

   {Our  numerical simulations will be carried out by using Matlab.   In order to show the region of attraction intuitively in a plane, we will consider the simple case consists of two oscillators.  Then, the dynamics is given by
\begin{align*}
\begin{aligned}
m_1\ddot{\theta}_1+d_1\dot{\theta}_{1}& = \Omega_{1} +   a_{12} \sin(\theta_{2} - \theta_{1}),\\ 
m_2\ddot{\theta}_2+d_2\dot{\theta}_{2}& = \Omega_{2} +   a_{21} \sin(\theta_{1} - \theta_{2}).
\end{aligned}
\end{align*}  To reduce the dimension of variables, we assume that  the initial frequencies are determined by initial phases in the following way:
\[
d_1\omega_1(0)=\Omega_1+a_{12}\sin(\theta_2(0)-\theta_1(0)), \quad d_2\omega_2(0)=\Omega_2+a_{21}\sin(\theta_1(0)-\theta_2(0)).
\]
Note that the dampings can be inhomogeneous, thus the two-oscillator system cannot be written as a single equation of $\theta:=\theta_1-\theta_2$.
We set  the parameters  $m_i$ and $d_i$  by using  random data which are  uniformly distributed in the following  way:
\[m_i\in (0.10, 0.15), \quad d_i\in(0.30, 0.40),\]
 and set the symmetric coupling strength as $a_{12}=a_{21}=0.2.$ We have $L_*=1$.}

  \subsection{Varying $\e$.}   {In this part,  we set   the range of phases  as \[D_0=\pi/4\] which fits the condition \eqref{assume}. Then we can calculate the parameters $\mathcal R_0, \lambda,  D_0, C_0, C_1, C_l, \tilde C_l$, and the interval for the possible location of  the positive coefficient $\varepsilon$ for the energy functional \[
\wt{\me}[\theta,\omega] :=\e\sum_{i=1}^2  d_i (\theta_i-\theta_c)^2+ 2\e\sum_{i=1}^2 m_i (\theta_i - \theta_c)\omega_i+ \sum_{i=1}^2  m_i \omega_i^2\quad \mbox{with} \quad \theta_c = \frac12 \sum_{i=1}^2  \theta_i.
\]
The natural frequencies $\Omega_i,i=1,2$ are randomly  chosen as   sufficient small data which have mean 0  and satisfy the condition \eqref{assumpB}.
Then  we can finally illustrate the region of attraction in $[0,\pi] \times [0,\pi]$, which is shown in Fig. 1 (a). The region of attraction is registered by the dark color.  For different choices of  admissible coefficients $\varepsilon$ satisfying {\bf (H3)}, we illustrate the boundary of the region in Fig. 1 (b). The different choices of $\e$ are registered by the different colors.  
We observe that the smaller choice of $\varepsilon$ produces a relative larger region of attraction.}


\begin{figure}
  \centering
  \subfigure[]{\includegraphics[scale=0.22]{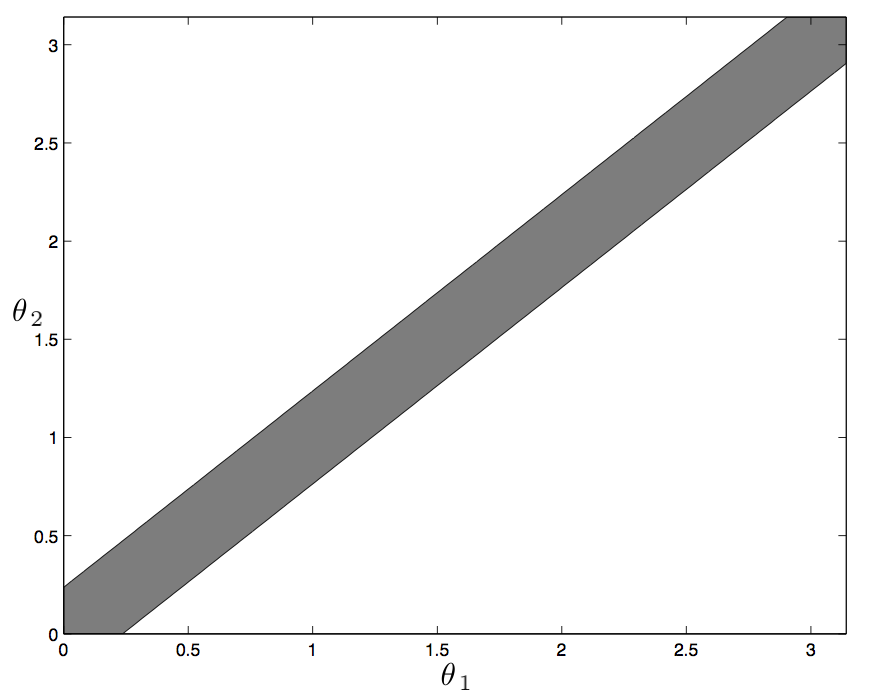}
  }
  \subfigure[]{
    \includegraphics[scale=0.22]{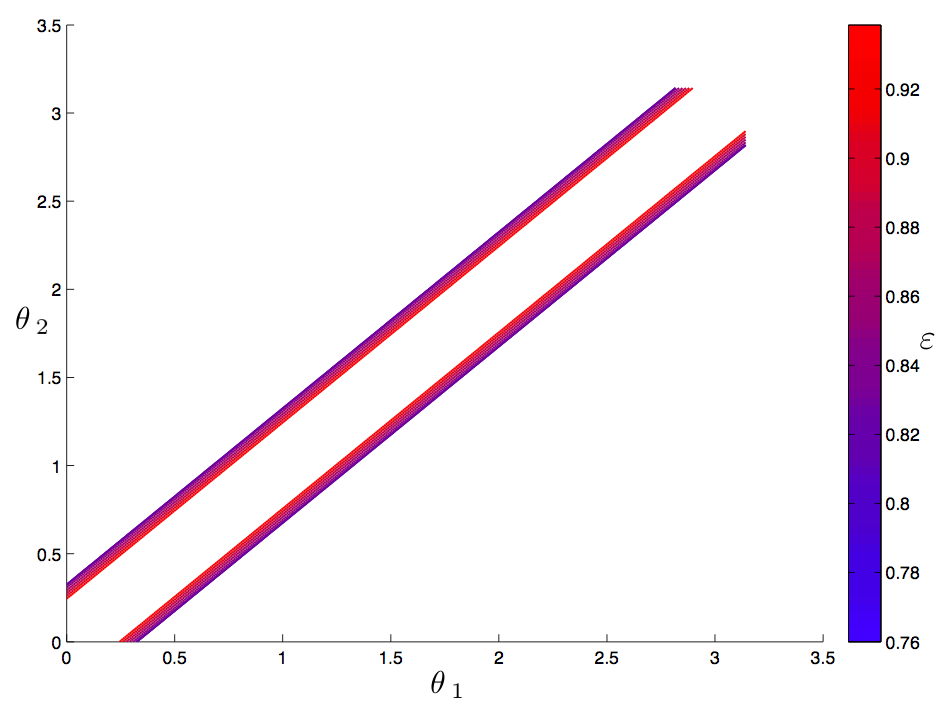}
    }
  \caption{(a): The region of attraction for a special choice of admissible $\e$.  (b): The boundary of region of attraction depending on $\e$. }
\end{figure}

\subsection{Varying $D_0$}
   {In Fig. 2, we try to illustrate the estimated  region of attraction for different  choices of constant $D_0\in (0,\pi)$, which needs to fit the condition \eqref{assume}.  We choose 18 numbers in $(0,\pi)$:
 \[
 \frac{\pi}{19},\, \frac{2\pi}{19},\,\frac{3\pi}{20},\, \dots, \frac{18\pi}{19},
 \]
 and use the restriction \eqref{assume} to find out the admissible ones. Simple computation indicates that all numbers in $[ \frac{\pi}{19},  \frac{9\pi}{19}]$ fit the condition \eqref{assume}.
 Then  we carry out the simulation using the admissible ones.    In view of Fig. 1, we choose $\e$  as  the smallest one   among the admissible choices of $\e$.  Fig. 2 shows the result depending on the values of $D_0$, which indicates that the larger choice of $D_0$ produces a larger region.}

\begin{figure}
  \centering
  \subfigure[]{\includegraphics[scale=0.22]{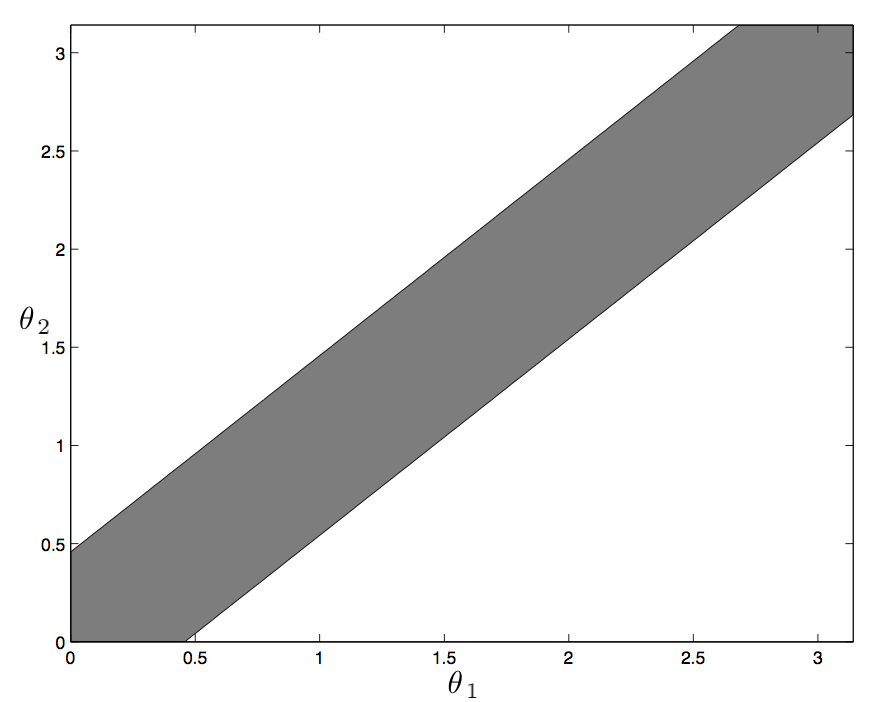}
  }
  \subfigure[]{
    \includegraphics[scale=0.3]{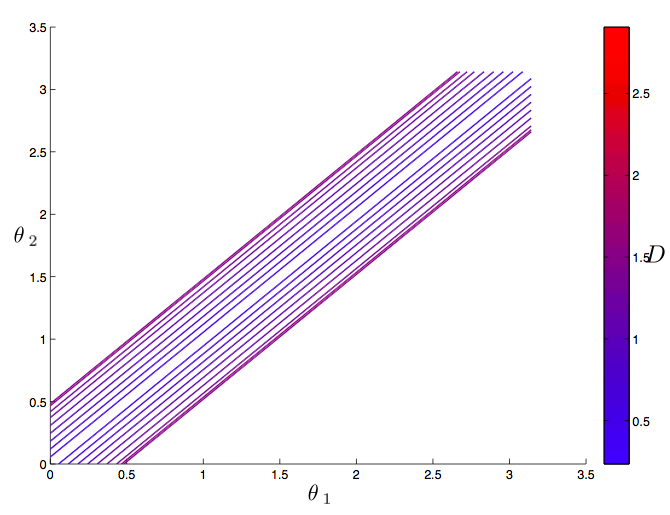}
    }
  \caption{(a): The region of attraction for a special choice of admissible $D_0$.  (b): The boundary of region of attraction depending on $D_0$. }
\end{figure}

\subsection{Conservativeness}
  {We acknowledge that our result is conservative in the sense that the  framework is only sufficient for the phase locking behavior, for example, the  presented estimate on  the region of attraction.  We   do some simulations, see Fig. 3, to illustrate this. We use the same parameters as in the simulation for Fig. 1. For the initial phases, we chose $(\theta_1,\theta_2) = (3,1)$ which does not fit  any region   shown in Figs. 1-2.  The employed numerical method is a classical fourth order Runge-Kutta one using the built-in {\it ode45} Matlab command. The simulation in Fig. 3 shows that the frequencies are   synchronized at an exponential rate, so the phase converges  to a phase-locked state. This suggests a future problem to improve the estimate of the  region of attraction.}

\begin{figure}
  \centering
  \subfigure[]{\includegraphics[scale=0.5]{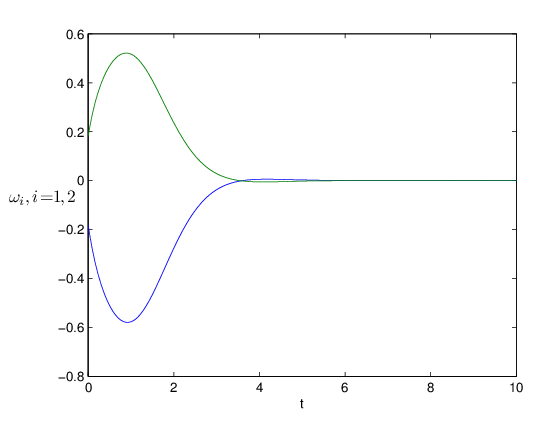}
  }
  \subfigure[]{
    \includegraphics[scale=0.5]{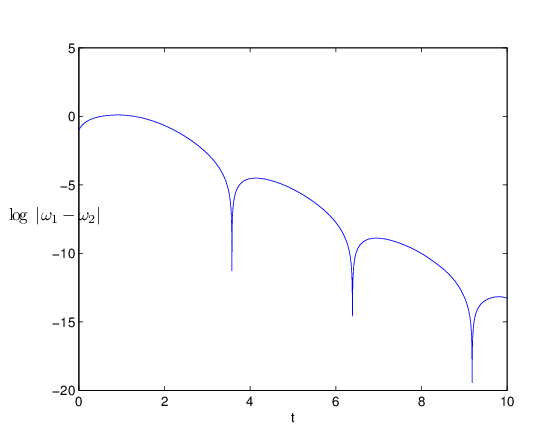}
    }
  \caption{(a): The evolution of $\omega_i,i=1,2$. (b): The evolution of $\log   |\omega_1 - \omega_2|$.}
\end{figure}

%
%
%
%
%
%
%

\section{Conclusion}
\setcounter{equation}{0} In this paper, we studied the synchronization and transient stability of the power grids on connected networks with inhomogeneous dampings. As mentioned before, the central problem  for the transient stability is to identify the region of attraction of the synchronous states, which was considered actually very rare. In \cite{L-X-Y1},  a special case of the power network model was considered: the damping is homogeneous and the underlying graph has a diameter less than or equal to 2. This is very strict in real applications for power grids.  {Moreover}, the analysis  in \cite{L-X-Y1}, based on the phase diameter,  heavily relied on these assumptions and cannot be extended to general cases.  In the present work, we employed the energy method to overcome the difficulty and obtained the desired estimate for this problem in the general case. 
  {Simulations  are provided to give some comparison on the different choices of parameters $D_0$ and $\e$, for a special setting of the simple network with two oscillators.
In view of the potential application in engineering, the quantitative improvement of the  estimate,  including    the parametric condition and  the region of attraction, would be an interesting future problem. The heterogeneity of the parameters and/or general connectivity mean that the method of studying the phase difference cannot work well, while our estimate gives a way   to overcome these difficulties. It is reasonable to expect a refined energy functional and a better energy estimate to improve the current result.}


\section*{Acknowledgments}
 Z. Li was   supported by   973 Program (2012CB215201),  National Nature Science Foundation of China (11401135), and the Fundamental Research Funds for the Central Universities (HIT.BRETIII.201501 and HIT.PIRS.201610).  Y.-P. Choi was partially supported by Basic Science Research Program through the National Research Foundation of Korea (NRF) funded by the Ministry of Education, Science and Technology (2012R1A6A3A03039496), Engineering and Physical Sciences Research Council (EP/K00804/1), and ERC-Starting grant HDSPCONTR ``High-Dimensional Sparse Optimal Control''. Y.-P. Choi is also supported by the Alexander Humboldt Foundation through the Humboldt Research Fellowship for Postdoctoral Researchers.

\end{document}